\documentclass[11pt,draft]{amsproc}
\usepackage{amstext,amsmath,amssymb,amsthm}
\usepackage{latexsym}
\usepackage{exscale}
\usepackage{color}

\numberwithin{equation}{section}

\newcommand\h[1]{\mkern2mu\widehat{\mkern-2mu#1}\mkern2mu}
\newcommand\lcheck{\mkern-2mu\raisebox{-1ex}{\huge$\check{}$}}

\newcommand\R{\mathbb{R}}
\newcommand\N{\mathbb{N}}

\newcommand\la{\langle}
\newcommand\ra{\rangle}

\newcommand{\G}{\mathbf{G}}

\newcommand\supp{\operatorname{supp}}

\renewcommand\div{\operatorname{div}}

\newtheorem{Thm}{Theorem}[section]
\newtheorem{Lemma}[Thm]{Lemma}
\newtheorem{Cor}[Thm]{Corollary}

\theoremstyle{remark}
\newtheorem{Rem}{Remark}[section]

\begin{document}

\title[weighted gradient inequalities and unique continuation] {Weighted gradient inequalities and unique continuation problems}
\author{Laura De Carli}
\address{L.~De Carli, Florida International University,
Department of Mathematics,
Miami, FL 33199, USA}
\email{decarlil@fiu.edu}
\author{Dmitry~Gorbachev}
\address{D.~Gorbachev, Tula State University,
Department of Applied Mathematics and Computer Science,
300012 Tula, Russia}
\email{dvgmail@mail.ru}
\author{Sergey~Tikhonov}
\address{S. Tikhonov, ICREA, Centre de Recerca Matem\`{a}tica, and UAB\\
Campus de Bellaterra, Edifici~C
08193 Bellaterra (Barcelona), Spain.}
\email{stikhonov@crm.cat}
\thanks{D.~G. was supported by
the Russian Science Foundation under grant 18-11-00199. S.~T. was partially
supported by MTM 2017-87409-P, 2017 SGR 358, and by the CERCA Programme of the
Generalitat de Catalunya.}
\begin{abstract}
We use Pitt inequalities for the Fourier transform to prove the following
weighted gradient inequality
$$
\|e^{-\tau\ell(\cdot)} u^{\frac 1q} f\|_q\leq c_\tau\| e^{-\tau\ell(\cdot)} v^{\frac 1p}\, \nabla f\|_p, \quad f\in C^\infty_0(\R^n).
$$
This inequality is a Carleman-type estimate that yields unique continuation
results for solutions of first order differential equations and systems.
\end{abstract}
\maketitle

\section{Introduction}


The main purpose of this paper is to prove that the following weighted Sobolev
gradient inequality holds for every linear function $\ell\colon \R^n\to\R$,
every $f\in C^\infty_0(\R^n)$ and every $\tau\ge 0 $, with suitable weights
$u$, $v$ and exponents $1<p,\, q<\infty$.
\begin{equation}\label{e-weighted-grad-tau}
\|e^{-\tau\ell(\cdot)} u^{\frac 1q} f\|_{q}\leq c_\tau\| e^{-\tau\ell(\cdot)} v^{\frac 1p}\, \nabla f\|_{p}
\end{equation}
Here, $c_\tau $ is a finite constant that may depend on $\tau$ but does not depend
on $\ell$ and $f$. We have denoted with $\| f \|_r=\bigl(\int_{\R^n}
|f(x)|^r\,dx\bigr)^{\frac 1r}$ the norm in $L^r(\R^n)$ and
  with $\la x,y\ra= x_1y_1+\cdots + x_ny_n$ and $|x|= \la x,\, x\ra^{\frac12}$ the standard inner product and norm in $\R^n$.

When $\tau>0$, we prove in Theorem \ref{T1-grad-heinig} that
$c_\tau=\max{}(\tau^{-1},1)C$; here and throughout the paper, $C$ denotes a
generic constant that depends only on non-essential parameters, i.e.
$C=C_{u,v,p,q,n}$. In particular, $c_\tau= C $ when $\tau\ge 1$.
%
Inequalities like \eqref{e-weighted-grad-tau} are often called \textit{Carleman
inequalities} in literature. In Sections~\ref{sec-unic-prob} and
\ref{sec-lin-sys} we will discuss Carleman inequalities and their connection
with unique continuation problems and we will prove new unique continuation
results for systems of partial differential equations and inequalities.

When $\tau=0$ in \eqref{e-weighted-grad-tau}, we obtain a standard
 weighted Sobolev gradient inequality (also called \textit{weighted Poincar\'{e}-Sobolev inequality})
\begin{equation}\label{e1-weighted-grad}
\| u ^{\frac 1q} f\|_q\leq c_0\| v ^{\frac 1p}\, \nabla f\|_p, \quad f\in C^\infty_0(\R^n).
\end{equation}
 These inequalities have deep applications in partial differential equations. For example,
 the case $p = 2 < q$ of \eqref{e1-weighted-grad} arises in Harnack's inequality and regularity
 estimates for degenerate second order differential operators in divergence form.
 They also have applications in the study of the stable solutions of
 the Laplace and the $p$-Laplace operators in
the Euclidean space, the Laplace-Kohn operator in the Heisenberg group, the
sub-Laplace operator in the Engel group, etc.;  see e.g. \cite{SW, FV, YL} and
the references cited in these papers; see also \cite{D}.

\medskip
Conditions on the weights $u$ and $v$ and the exponents $p$, $q$ for which \eqref{e1-weighted-grad} holds have been investigated by several authors.
 The most natural approach to study \eqref{e1-weighted-grad} is based on the
 following pointwise inequality (see e.g. \cite{FKS, Perez})
\[
|f(x)|\leq C I_1(|\nabla f|)(x),\quad x\in \R^{n},
\]
where $I_\alpha \phi (x)=\int_{\R^n}\frac{\phi(y)}{|x-y|^{n-\alpha}}\,dy$,
$\alpha<n$, is the \textit{Riesz potential}. This inequality follows from
the classical Sobolev integral representation and is proved e.g. in
\cite{Maz}.

  If the weighted inequality
\begin{equation}\label{uI-vf-ineq}
  \|u^{\frac 1q}I_1 f\|_q \leq C\|v^{\frac 1p} f\|_p
\end{equation}
  holds for the weights $u$ and $v$,
we also have %
\[
\|u^{\frac 1q}f\|_q \leq C \|u^{\frac 1q} I_1(|\nabla f|)\|_q \leq C \|v^{\frac
1p} |\nabla f|\|_p.
\]
  E. Sawyer obtained in \cite{Sawyer} a complete characterization of
the weights $u$ and $v$ for which the gradient
inequality \eqref{uI-vf-ineq} holds with $p\le q$. However, in some cases,
the conditions in \cite{Sawyer} are difficult to verify. When $p=q=2$, a full
characterization of the weights for which \eqref{e1-weighted-grad} holds is
also in \cite{LN}, but also the conditions in this paper are difficult to
verify.

H.~P.~Heinig showed in \cite{Heinig} that weighted norm inequalities for the
Fourier transform (or: \textit{Pitt-type inequalities}) in the form of
 \begin{equation}\label{e1-Genpitt}
\|\h f \,u^{\frac 1q}\|_q\leq C\| f\, w^{\frac 1p}\|_p, \quad f\in
C^\infty_0(\R^n),
\end{equation}
 can be used to prove weighted gradient inequalities. The Fourier transform is defined as $\h f(y)=\int_{\R^n} f(x)e^{-i\la x,\, y\ra}dx$.

 To prove \eqref{e1-weighted-grad} from \eqref{e1-Genpitt}, we observe that
$
\widehat{I_\alpha f}(y)= c_\alpha|y|^{-\alpha}\h f(y),
$
where $c_{\alpha}$ is an explicit constant; we can see at
once that \eqref{uI-vf-ineq} is equivalent to
\[
\|u^{\frac 1q}(|y|^{-1}\h f)\;\lcheck{}
\;\|_q \leq C\|v^{\frac 1p} f\|_p
\]
where $\;\lcheck\;$ denotes the inverse Fourier transform.
We can
apply Pitt's inequality twice (with a suitable weight $w$ and an exponent $\gamma\in (1,\infty)$)
to obtain
\[
\|u^{\frac 1q}(|y|^{-1}\h f)\;\lcheck{}
\;\|_q \leq C \|w^{\frac 1\gamma}
|y|^{-1}\h f\|_\gamma\leq C\|v^{\frac 1p} f\|_p.
\]
Taking $w=|y|^{\gamma}$ and $\gamma=q$ and assuming conditions on the weights
that ensure that both Pitt's inequalities hold  we obtain the main theorem in
\cite{Heinig}, which was proved differently; see Theorem~\ref{T1- Pitt-heinig}
in Section~\ref{sec-grad-heinig}.

\subsection{ Main results}

Throughout this paper, we will {often} write $A\lesssim B$ when $A\le CB$ with a
constant $C>0$.
 We will also write $A\asymp B$ when there exists a constant $C>0$,
called the \textit{constant of equivalence}, such that $C^{-1}A\le B\le CA$.
  As usual, we let~$g^*$ be the non-increasing rearrangement of $g$. We let $p'=\frac{p}{p-1}$ be the dual exponent of $p\in (1, \infty)$.

\medskip
Our main result can be stated as follows.
\begin{Thm}\label{T1-grad-heinig}
Let $u\not\equiv 0$ and $v\not\equiv +\infty$ be weights on
 $\R^{n}$, $n\ge 1$.

\medskip
\textup{(a)} Let $1<p\le q<\infty$. If there exists $\gamma>0$ that satisfies
$\max{}(p,p')\le \gamma\le q$,
 for which
\begin{equation}\label{u*-cond}
\begin{cases}
\displaystyle
A_{u}^{q}(0):=
\sup_{s>0}s^{1-q(\frac{1}{\gamma'}-\frac{1}{n})}u^{*}(s)<\infty,\qquad \mbox{$\frac{1}{n}<\frac{1}{\gamma'}\le
\frac{1}{n}+\frac{1}{q}$},
\\
\displaystyle A_{u}^{q}(\tau):=
 \sup_{s>0}\ \int_{0}^s u^*(t)\,dt
\Bigl(\int_{0}^{\frac{1}{s}}
(t+\tau^{n})^{-\frac {\gamma'}{n} }\,dt\Bigr)^{\frac q{\gamma'}} <\infty,\quad \tau> 0,
\end{cases}
\end{equation}
and
\begin{equation}\label{v*-cond}
A_{v}^{p}:= \sup_{s>0}s^{\frac{p}{\gamma'}-1}(1/v)^{*}(s) <\infty,
\end{equation}
the inequality
\begin{equation}\label{e2-weighted-grad}
\|e^{-\tau \ell(\cdot)}u^{\frac 1 q}f\|_q\le c_\tau\|e^{-\tau \ell(\cdot)}v^{\frac 1p}\,\nabla
f\|_p, \quad f\in C^\infty_0(\R^{n}),
\end{equation}
holds for every $\tau\ge 0$ and every linear function $\ell(x)=\la
\mathbf{a},x\ra+b$, $\mathbf{a}\in \R^{n}$, $|\mathbf{a}|=1$, $b\in \R$, with
the constant
\begin{equation}\label{ctau-bound}
c_\tau=C A_{u}(\tau)A_{v},
\end{equation}
where $C=C_{p,q,\gamma,n}$ is some positive constant. Moreover,
\begin{equation}\label{Au-bound}
A_{u}(\tau)\le \max{}(\tau^{-1},1)A_{u}(1),\quad \tau>0.
\end{equation}

\medskip \textup{(b)} Let $1<q<p<\infty$. If there exists $\gamma>0$ that
satisfies
\[
\begin{cases}
\frac{n}{n-1}<\gamma\le q, & \tau=0,\\ 1<\gamma\le q, & \tau>0,
\end{cases}
\]
for which \eqref{u*-cond} holds and
\[
\tilde{A}_{v}^{r}:= \int_0^\infty s^{-\frac{r}{\gamma}-1}
\Bigl(\int_0^{s}(1/v)^*(t)^{\frac{1}{p-1}}\,dt\Bigr)^{\frac{r}{p'}}\,ds <\infty,
\]
with $\frac 1 r=\frac 1 \gamma-\frac 1 p$, the inequality
\eqref{e2-weighted-grad} holds with the constant
\begin{equation}\label{ctau-bound-tilde}
c_\tau=CA_{u}(\tau)\tilde{A}_{v}.
\end{equation}
\end{Thm}

\medskip
\begin{Rem}
When $\tau=0$ and $\gamma= q$ we obtain Theorem~2.4 in \cite{Heinig} 
with simplified conditions on $u$ and $v$. The proof of
Theorem~\ref{T1-grad-heinig} shows that the assumptions $\frac{1}{\gamma'}\le
\frac{1}{n}+\frac{1}{q}$ and $p'\le \gamma$ are to rule out the trivial
weights $u\equiv 0$ and $v\equiv +\infty$.
\end{Rem}

\begin{Rem}\label{rem-ctau-bound}
For the applications of Theorem \ref{T1-grad-heinig} it is
important to have the uniform boundedness of $c_{\tau}$ as $\tau\to \infty$.
From \eqref{ctau-bound}, \eqref{ctau-bound-tilde}  and \eqref{Au-bound},  we have $c_{\tau}\le c_{1}\asymp A_{u}(1)$ whenever
$\tau\ge 1$; thus, to prove the
boundedness of $c_\tau$, it is sufficient to verify that $A_{u}(1)<\infty$.
\end{Rem}

 \begin{Rem}
 It is interesting to compare our weighted gradient inequalities with those proved by G. Sinnamon
in \cite{Sinnamon}. In that paper, a weighted norm inequality in the form of
 \begin{equation}\label{e1-sinnamon}
\|f u^{\frac 1q}\|_q\leq C\|\la \nabla f,\, x\ra w^{\frac 1p}\|_p, \quad f\in C^\infty_0(\R^n)
\end{equation}
is considered. If we denote with $\partial_r f= \la \frac{x}{|x|},\, \nabla f\ra $ the
radial derivative of $f$, the inequality \eqref{e1-sinnamon} is equivalent to
\begin{equation*}\label{e1-radial-sinnamon}
\|f u^{\frac 1q}\|_q\leq C\|\, |x| w ^{\frac 1p} \partial_r f \|_p, \quad f\in
C^\infty_0(\R^n),
\end{equation*}
and implies \eqref{e1-weighted-grad} with $v=|x|^pw $.

In \cite[Theorem 4.1]{Sinnamon}, \eqref{e1-sinnamon} is only proved for $p=q$ and
$q<p$ under some conditions on $u$, $w$; moreover,
in \cite[Theorem~3.4]{Sinnamon} it is proved that when $1\leq p<q<\infty$ and
the weight $w$ is locally integrable on $\R^n$, the inequality
\eqref{e1-sinnamon} holds for every $f\in C^\infty_0(\R^n)$ if and only if
$u\equiv 0$ a.e.

When $f$ is radial, $\nabla f(x) =\frac{x}{|x|}\,\partial_r f(x)$, and so
$|\nabla f(x)|=|\partial_r f(x)|$. Thus, our Theorem \ref{T1-grad-heinig}
yields \eqref{e1-sinnamon} for radial functions with a nontrivial weight $u$
and with $w=|x|^{-p}e^{-p\tau\ell(x)}v$. We proved in Corollary
\ref{Cor-power-weight} below that we can consider piecewise power weights
$v=|x|^{(\beta_1,\, \beta_2)}$, with $0\le \beta_{1}\le
n\bigl(\frac{p}{\gamma'}-1\bigr)$ (see definition \eqref{def-power-f}). For
example, if $\beta_{1}=n\bigl(\frac{p}{\gamma'}-1\bigr)$, then $w$ is locally
integrable for $\frac{1}{n}<\frac{1}{\gamma'}$ because $-p+\beta_{1}>-n$. We
remark that the counterexample in \cite[Theorem~3.4]{Sinnamon} is not radial.
\end{Rem}

 \begin{Rem}\label{rem-u-v-a-ineq}
The inequality \eqref{e2-weighted-grad} is equivalent to
\begin{equation}\label{u-v-a-ineq}
\|u^{\frac 1q}f\|_q\le c_\tau \|v^{\frac 1 p}\,(\tau \mathbf{a}f+\nabla f)\|_p.
\end{equation}
\end{Rem}

To see this, it is enough to use the substitution $f_{1}=e^{-\tau \ell(\cdot)}f$ and $\nabla (e^{\tau \ell(\cdot)}f_{1})=e^{\tau
\ell(\cdot)}(\tau \mathbf{a}f_{1}+\nabla f_{1})$.

\medskip
Let $\beta_1$, $\beta_2\in\R$; we define the piecewise power function $t\mapsto t^{(\beta_{1},\beta_{2})}$ as follows:
\begin{equation}\label{def-power-f}
t^{(\beta_{1},\beta_{2})}:=\begin{cases} t^{\beta_{1}},& 0<t\le 1,\\
t^{\beta_{2}},& t\ge 1. \end{cases}
\end{equation}
In the following corollary of Theorem \ref{T1-grad-heinig} we consider the
important case of piecewise power weights.

\begin{Cor}\label{Cor-power-weight}
Let $1<p\leq q<\infty$; let $\gamma>0$ that satisfies $\max{}(p,p')\le
\gamma\le q$ and $\frac 1n<\frac{1}{\gamma'}\leq \frac
1n+\frac 1q$. 

With the notation and the assumptions of Theorem
\textup{\ref{T1-grad-heinig}~(a)}, the inequality \eqref{e2-weighted-grad}
holds with $u(x)=|x|^{(-\alpha_{1},-\alpha_{2})}$,
$v(x)=|x|^{(\beta_{1},\beta_{2})}$, with $\alpha_j,\,\beta_j\ge 0$, provided
that
\begin{equation}\label{cond-alpha}
\alpha_{1}\le n\Bigl(1-\frac{q}{\gamma'}+\frac{q}{n}\Bigr),\quad
\begin{cases}
\alpha_{2}\ge n\bigl(1-\frac{q}{\gamma'}+\frac{q}{n}\bigr) & \text{when } \ \tau=0,\\
\alpha_{2}\ge 0 & \text{when }\ \tau>0,
\end{cases}
\end{equation}
\begin{equation}\label{cond-beta}
\beta_{1}\le n\Bigl(\frac{p}{\gamma'}-1\Bigr),\quad \beta_{2}\ge
n\Bigl(\frac{p}{\gamma'}-1\Bigr).
\end{equation}
In particular, for power weights $u(x)=|x|^{-\alpha}$, $v(x)=|x|^{\beta}$ the inequality \eqref{e2-weighted-grad} holds if
\[
\begin{cases}
\alpha=n\bigl(1-\frac{q}{\gamma'}+\frac{q}{n}\bigr)\ge 0 & \text{when }\ \tau=0,\\[1.5ex]
0\le \alpha\le n\bigl(1-\frac{q}{\gamma'}+\frac{q}{n}\bigr) &\text{when }\ \tau>0,
\end{cases}\qquad
\beta=n\Bigl(\frac{p}{\gamma'}-1\Bigr)\ge 0.
\]
Moreover, the conditions
\begin{equation}\label{ab-n-cond}
\begin{cases}
\frac{\alpha}{q}+\frac{\beta}{p}=n\bigl(\frac{1}{q}-\frac{1}{p}\bigr)+1 &
\text{when }\ \tau=0,\\[1.5ex]
\frac{\alpha}{q}+\frac{\beta}{p}\le n\bigl(\frac{1}{q}-\frac{1}{p}\bigr)+1 &
\text{when }\ \tau>0,
\end{cases}
\end{equation}
are necessary for the validity of \eqref{e2-weighted-grad}.
\end{Cor}

Letting $\tau=\alpha=\beta=0$, $1<p<n$, and $\gamma=q$ in (\ref{ab-n-cond}), we
obtain $q=\frac{np}{n-p}$
and Corollary \ref{Cor-power-weight} yields the classical Sobolev inequality
$\|f\|_q\le C \|\nabla f\|_p$; see also \cite[Corollary~2.5]{Heinig}.

When $\tau=0$, we obtain the inequality
\[
\Bigl(\int_{\R^n}|f|^q |x|^{(-\alpha_{1},-\alpha_{2})}\,dx\Bigr)^{\frac 1 q}\le C
\Bigl(\int_{\R^n}|\nabla f|^p |x|^{(\beta_{1},\beta_{2})}\,dx\Bigr)^{\frac 1p},
\]
which was proved by Maz'ya \cite{Maz} and Caffarelli,
Kohn, and Nirenberg \cite{caff} for power weights.
In \cite[Sect.~2.1.6]{Maz} it was proved that if $1<p<n$, $p\le q\le
\frac{pn}{n-p}$, and
$-\frac{\alpha}{q}=\frac{\beta}{p}-1+n\left(\frac{1}{p}-\frac{1}{q}\right)>-\frac{n}{q}$,
then
\begin{equation}\label{e-powerw}
\Bigl(\int_{\R^n}|f|^q |x|^{-\alpha}\,dx\Bigr)^{\frac 1 q}\le C
\Bigl(\int_{\R^n}|\nabla f|^p |x|^{\beta}\,dx\Bigr)^{\frac 1p}.
\end{equation}
In \cite[Lemma~2.1]{Ho}, this inequality was proved 
for $n\ge
2$, $1<p<+\infty$, $0\le
\frac{1}{p}-\frac{1}{q}=n\left(1-\frac{\beta}{p}-\frac{\alpha}{q}\right)$ and
$-\frac{n}{q}<-\frac{\alpha}{q}\le \frac{\beta}{p}$.
Note that the conditions in \cite{Maz} and \cite {Ho} are the same, except for the extra condition $p<n$ in
\cite{Maz}.

From Corollary~\ref{Cor-power-weight} with $\tau=0$ we have that
$\alpha=n\bigl(1-\frac{q}{\gamma'}+\frac{q}{n}\bigr)\ge 0$,
$\beta=n\bigl(\frac{p}{\gamma'}-1\bigr)\ge 0$, where $\max{}(p,p')\le \gamma\le
q$ and $\frac 1n<\frac{1}{\gamma'}\leq \frac 1n+\frac 1q$. These inequalities
imply
$\frac{1}{p}-\frac{1}{q}=n\bigl(1-\frac{\beta}{p}-\frac{\alpha}{q}\bigr)$,
$-\frac{n}{q}<-\frac{\alpha}{q}\le \frac{\beta}{p}$, but we also have to assume
$\alpha\ge 0$, $\beta\ge 0$ because of our method of the proof.

It is interesting to observe that the best constant in the inequality
\eqref{e-powerw} has been evaluated in \cite{YL} and also in \cite{FV} for
special values of $\alpha$ and $\beta$.

\subsection{Unique continuation}
Our Theorem \ref{T1-grad-heinig} can be used to prove unique continuation
results for  weak  solutions (also called {\it solutions in distribution sense}) of systems
of differential equations and inequalities; see Section~\ref{sec-unic-prob} for
definitions and preliminary results.

We consider solutions in weighted Sobolev spaces of distributions: given a
domain $D\subset \R^n$, we let $W^{m, p, v }_0(D)$ be the closure of
$C^\infty_0(D)$ with respect to the norm
\[
\|f\|_{W^{m,p,v}(D)}= \sum_{|\alpha|=0}^m \|v^{\frac 1p} \partial^\alpha_x
f\|_{L^p(D)}
\]
where $\alpha=(\alpha_1, \, \dots,\, \alpha_n)\in \N^n$ and the
$\partial^\alpha_x f = \partial^{\alpha_1}_{x_1}\cdots
\partial^{\alpha_n}_{x_n}f$ are the partial derivatives of $f$. In Section~\ref{sec-unic-prob}
we prove the following

\begin{Thm}\label{T1-UC-grad}
Let $p$, $q$, $\gamma$, $u$ and $v$ be as in
Theorem~\textup{\ref{T1-grad-heinig}~(a)}. Let $\frac 1r=\frac 1p-\frac 1q$.
Let $f\in W^{1,p,v}_0(\R^n)$ be a solution of the differential inequality
\begin{equation}\label{ineq-V-nable}
|\nabla f|\leq V |f|
\end{equation}
with $V\in L^r (\supp f, \, v^{\frac rp} u^{-\frac rq}\,dx)$.
If, for some linear function $\ell\colon \R^n\to\R$, we have that $\supp f
\subset \{x\colon \ell(x)\geq 0\} $, necessarily $f\equiv 0$.
\end{Thm}

Note that the condition $V\in L^r (\supp f, \, v^{\frac rp}
u^{-\frac rq}\,dx)$ follows from either $V\in L^r (\R^{n}, \, v^{\frac rp}
u^{-\frac rq}\,dx)$ if $\supp f$ is unbounded, or from $V\in L_\mathrm{loc}^r
(\R^{n}, \, v^{\frac rp} u^{-\frac rq}\,dx)$ if $  f$ has compact support.
In particular, for power weights $u$, $v$ as in
Corollary~\ref{Cor-power-weight}, the differential inequality
\eqref{ineq-V-nable} does not have solutions with compact   support if
$V\asymp |x|^{-1+\epsilon}$ for some $\epsilon>0$; see
Remark~\ref{rem-Vepsilon}.

\medskip

To prove Theorem \ref{T1-UC-grad} we use a method developed by T. Carleman in
\cite{C}. A brief discussion on unique continuation problems and Carleman's
method is in Sections~\ref{sec-unic-prob} and \ref{sec-lin-sys}.

\medskip When $D$ is measurable and $v$ is a suitable weight we consider the
Dirichlet problem
\begin{equation}\label{e-dir1}
\begin{cases}
-\div{}( v\, \nabla f |\nabla f|^{p-2}) =v\,Vf |f|^{p-2}, \\ f\in W_0^{1,p,v}(D),
\end{cases} \end{equation}
where $\div{}((g_1,\, \dots, \, g_n))= \partial_{x_1} g_1+\dots+\partial_{x_n}
g_n$ and the potential $V$ is in a suitable $L^r$ space. The operator $
\div{}( v\, \nabla f |\nabla f|^{p-2})$ is known as {\it weighted
$p$-Laplacian} in the literature (see e.g. \cite{FPR, K}) and is denoted
by $\Delta_{p}$ when $v\equiv 1$. The weighted $p$-Laplacian is nonlinear
when $p\ne 2$ and is linear when $p=2$.

When $v\equiv 1$, \eqref{e-dir1} can be compared to the Sturm-Liouville problem
in the form of $-\Delta_{p}f=(\lambda m-V)f|f|^{p-2}$ (see e.g. \cite{CQ}).
When $n=1$ and $p=2$ we have $-(vf')'=vVf$. This problem is related to the
classical Sturm-Liouville problem $-(vf')'=(\lambda w-q)f$.   See \cite{LS}.

\medskip
We prove the following

\begin{Thm}\label{T-UC-dir1}
Let $f \in W_0^{1,p, v}(D)$ be a solution of the Dirichlet problem
\eqref{e-dir1}. Let $V_+=\max \{V, 0\}$. Assume that $|V|^{\frac
1p}\in L^r (D, \, v^{\frac rp} u^{-\frac rq}\,dx)$, where $u$, $v$ are as in
Theorem~\textup{\ref{T1-grad-heinig}} and $\frac 1r=\frac 1p-\frac 1q$. Then,
either
\begin{equation*}\label{mainab}
c_0\|u^{-\frac 1q}v^{\frac 1p} \, V_+^{\frac 1p}\|_{L^r(D)} \ge 1,
\end{equation*}
where $c_0 $ is as in \eqref{e1-weighted-grad}, or $f\equiv 0$ in $D$.
\end{Thm}

Thus, the Dirichlet problem \eqref{e-dir1} has the unique solution $f\equiv 0$
if the weighted $L^r$ norm of $V_+^{\frac 1p}$ on $ D $ is small enough.

\medskip To the best of our knowledge, the method of proof of Theorem
\ref{T-UC-dir1} has been used for the first time in \cite{DH}; it is extensively used
in \cite{DEHL} and \cite{EHL}.

\section{Proof of Theorem \ref{T1-grad-heinig}}\label{sec-grad-heinig}

 In this section we prove our main theorem and a few corollaries.

\subsection{Preliminary results}

 We will use the following theorem due to
 Heinig \cite{H}, Jurkat-Sampson \cite{JS}, and Muckenhoupt \cite{M1}.

\begin{Thm}\label{T1- Pitt-heinig} Let $n\ge 1$.
If $1<p\le q<\infty$ and the weights $u$ and $w$ satisfy
\begin{equation*}\label{1e-weight-cond}
\sup_{s>0}\Bigl(\int_{0}^{\frac 1s} u^*(t )\,dt\Bigr)^{\frac{1}{q}}
\Bigl(\int_{0}^{\frac{1}{s}} ((1/w)^*(t))^{\frac{1}{p-1}}\,dt\Bigr)^{\frac{1}{p'}}=:A_{1}<\infty,
\end{equation*}
or if $1<q<p<\infty$, and
\begin{equation}\label{2e-weight-cond}
  \sup_{s>0} \biggl(\int_0^\infty \Bigl(\int_0^ {\frac 1s} u^*(t) dt\Bigr)^{ \frac rq}
  \Bigl( \int_0^{s} ((1/w)^*(t))^{\frac{1}{p-1}}dt\Bigr)^{\frac{r}{q'}}
((1/w)^*(s))^{\frac{1}{p-1}}\,ds\biggr)^{\frac1r} $$$$ =:A_{2} < \infty
\end{equation}
where $ r=\frac{qp}{q-p}$,
then Pitt's inequality
\[
\|\h f \,u^{\frac 1q}\|_q\leq C_{j}\| f\, w^{\frac 1p}\|_p, \quad f\in
C^\infty_0(\R^n),\quad j=1,2,
\]
holds with $C_{j}\le C_{p,q,j}A_{j}$.
\end{Thm}

Recall that the non-increasing rearrangement of a measurable radially
decreasing function $f(x)=f_{0}(|x|)$ is defined as follows: let for
$\lambda>0$
\[
\mu_{f}(\lambda)=\mu\{x\colon |f(x)|>\lambda\}= \mu\{x\colon
|x|<f_{0}^{-1}(\lambda)\}=(f_{0}^{-1}(\lambda))^{n}V_{n},
\]
where $V_{n}$ is the volume of the unit ball $B^{n}=\{x\in \R^{n}\colon
|x|\le 1\}$. Then for $t>0$
\[
f^{*}(t)=\inf\{\lambda>0\colon \mu_{f}(\lambda)<t\}=
f_{0}((t/V_{n})^{\frac 1 n}).
\]

Note that the conditions on $u$ and $w$ are also necessary when $u$ and $w$
are radial, i.e., $u=u_0(|x|) $ and $w(x)=w_0(|x|)$, with $u_0(r)$
non-increasing and $w_0(r)$ non-decreasing. See \cite{H} and also
\cite[Theorem 1.2 ]{DGT2} for simpler and more general necessary conditions on
the weight $u$ and $w$.
We should also mention \cite[Theorem 2.1]{L} where a necessary condition
similar to that in \cite{H}, with $u$ replaced by a measure $d\mu$, was proved.



%
%
%

\medskip
 We also need the following

\begin{Lemma}\label{lem-psi}
Let $\psi\not\equiv 0$ be a non-increasing non-negative function; let
$\beta_{1},\beta_{2}>0$ and let $\beta_{2}'=\min{}(\beta_{2},1)$. If either
\[
A=\sup_{s>0}s^{(-\beta_{1},-\beta_{2})}\int_0^s \psi(t)\,dt<\infty,
\]
or
\[
B=\sup_{s>0}s^{(1-\beta_{1},1-\beta_{2}')}\psi(s)<\infty,
\]
then $\beta_{1}\le 1$ and $A\asymp B$.
\end{Lemma}

\begin{proof}
Assume $A<\infty$; then, for every $s>0$, we have that $\int_0^s \psi(t)\,dt\le
As^{(\beta_{1},\beta_{2})}$. Since $\psi$ is non-increasing, $s\psi(s)\le
\int_0^s \psi(t)\,dt$, so $\psi(s)\le As^{(\beta_{1}-1,\beta_{2}-1)}$. If
$\beta_{1}>1$, then $\displaystyle \lim_{s\to 0^+}\psi(s)=0$ and consequently
$\psi\equiv 0$; since we assumed $\psi\not\equiv 0$, necessarily $\beta_{1}\le
1$.

Furthermore, from
$\psi(s)\le \psi(1)$ for $s\ge 1$ we can see at once that $\psi(s)\lesssim As^{\beta_{2}'-1}$ and so
  $B\lesssim A$.

\medskip
If we assume $B<\infty$, for every $s>0$ we have that $\psi(s)\le Bs^{(\beta_{1}-1,\beta_{2}'-1)}$
As above we conclude that $\beta_{1}\le 1$. For $0<s\le 1$ we have $\int_0^s
\psi(t)\,dt\lesssim Bs^{\beta_{1}}$. If $s\ge 1$, then
\[
\int_0^s \psi(t)\,dt=\int_0^1 \psi(t)\,dt+ \int_1^s \psi(t)\,dt\lesssim
B+B\int_1^s t^{\beta_{2}'-1}\,dt\lesssim Bs^{\beta_{2}'}\le Bs^{\beta_{2}}.
\]
Thus, $\sup_{s\ge 1}s^{ -\beta_{2 }}\int_0^s \psi(t)\,dt \lesssim B$ and $A\lesssim B$.
\end{proof}

%
%
%
%

\subsection{Proof of Theorem \ref{T1-grad-heinig}}

We can assume $\ell(x)=\la \mathbf{a},x\ra$, $|\mathbf{a}|=1$, without loss of
generality.

\medbreak (a) \ Let $p\le \gamma\le q$.

\underline{Step~1.} \
For $\tau\ge 0$ and $\xi\in \R^{n}$, define
\[
w_{\tau}(\xi)=|\xi-i\tau
\mathbf{a}|^{\gamma}=(|\xi|^{2}+\tau^{2})^{\frac \gamma 2}.
\]
By Theorem~\ref{T1- Pitt-heinig} (a), the inequality
\begin{equation}\label{pitt-u-wt}
\Bigl(\int_{\R^{n}}|\h g(x)|^{q}u(x)\,dx\Bigr)^{\frac 1 q}\lesssim
A_{u,w_{\tau}}\Bigl(\int_{\R^{n}}w_{\tau}(\xi)|g(\xi)|^{\gamma}\,d\xi\Bigr)^{\frac
1{\gamma} }
\end{equation}
holds with
\[
A_{u,w_{\tau}}=\sup_{s>0}\Bigl(\int_{0}^s u^*(t)\,dt\Bigr)^{\frac{1}{q}}
\Bigl(\int_{0}^{\frac{1}{s}}
((1/w_{\tau})^*(t))^{\frac{1}{\gamma-1}}\,dt\Bigr)^{\frac{1}{\gamma'}}<\infty.
\]

The weight $w_{\tau}$ is radially increasing, so
\begin{equation*}\label{cond-wt}
(1/w_{\tau})^*(t)=((t/V_{n})^{\frac 2n}+\tau^{2})^{-\frac \gamma 2}\asymp (t+\tau^{n})^{-\frac \gamma n}
\end{equation*}
with the constant of equivalence independent of $\tau$. This implies
\begin{equation*}
\int_{0}^{\frac1s} ((1/w_{\tau})^*(t))^{\frac{1}{\gamma-1}}\,dt\asymp
\int_{0}^{\frac1s} (t+\tau^{n})^{-\frac {\gamma'}{n} }\,dt,\quad s>0.
\end{equation*}
Therefore, for $\tau\ge 0$,
\begin{equation}\label{cond-u-wt}
A_{u,w_{\tau}}^{q}\asymp \sup_{s>0}\int_{0}^s u^*(t)\,dt
\Bigl(\int_{0}^{\frac1s}
(t+\tau^{n})^{-\frac {\gamma'}{n} }\,dt\Bigr)^{\frac{q}{\gamma'}}=A_{u}^{q}(\tau).
\end{equation}

Since $(t+\tau^{n})^{-1}\le \max{}(\tau^{-n},1)(t+1)^{-1}$ for $t,\tau>0$, from
\eqref{cond-u-wt} we conclude that
\[
A_{u}(\tau)\le \max{}(\tau^{-1},1)A_{u}(1),\quad \tau>0.
\]

We can give a simple expression for $A_{u}^{q}(0)$. Observing that
$I:=\int_{0}^{1/s} t^{-\frac {\gamma'}{n} }\,dt$ is finite when $-\frac {\gamma'}{n} >-1$
or, equivalently,
$\frac{n}{n-1}<\gamma$, we have that $I\asymp s^{\frac {\gamma'}{n} -1}$. Therefore,
\eqref{cond-u-wt} can be rewritten as
\begin{equation}\label{C1q}
A_{u}^{q}(0)\asymp \sup_{s>0} s^{-q(\frac{1}{\gamma'}-\frac{1}{n})}\int_{0}^s
u^*(t)\,dt.
\end{equation}
By \eqref{C1q} and Lemma \ref{lem-psi} with
$\beta_{1}=\beta_{2}=q(\frac{1}{\gamma'}-\frac{1}{n})$, there holds that
$q({\frac{1}{\gamma'}-\frac{1}{n}})\le 1$ or $\frac{1}{\gamma'}\le
\frac{1}{n}+\frac{1}{q}$ and we can redefine $A_{u}^{q}(0)$ as follows.
$$
 A_{u}^{q}(0)=\sup_{s>0} s^{-q(\frac{1}{\gamma'}-\frac{1}{n})}\int_{0}^s
u^*(t)\,dt \asymp \sup_{s>0}s^{1-q(\frac{1}{\gamma'}-\frac{1}{n})}u^{*}(s).
$$

\smallbreak
\underline{Step~2.} \ Let $g(x)=e^{-\la \tau\mathbf{a},x\ra}f(x)$. Then $g\in
C^\infty_0(\R^{n})$ and
\begin{equation}\label{hat-g-f}
\h{g}(\xi)=\int_{\R^{n}}g(x)e^{-i\la\xi,\, x\ra}\,dx=
\int_{\R^{n}}f(x)e^{-i\la \xi,\, x\ra - \la \tau\mathbf a,\, x\ra}\,dx=\h{f}(\xi-i\tau\mathbf a).
\end{equation}

Since for $g\in C^\infty_0(\R^{n})$ the Fourier inversion formula holds,
\eqref{pitt-u-wt} and \eqref{cond-u-wt} imply
\begin{align}
\Bigl(\int_{\R^{n}}|g(x)|^{q}u(x)\,dx\Bigr)^{\frac 1 q}&\lesssim
A_{u}(\tau)\Bigl(\int_{\R^{n}}|\xi-i\tau\mathbf
a|^{\gamma}|\h{g}(\xi)|^{\gamma}\,d\xi\Bigr)^{\frac 1{\gamma} } \notag\\
&=A_{u}(\tau)\Bigl(\int_{\R^{n}}\bigl|(\xi-i\tau\mathbf
a)\h{f}(\xi-i\tau\mathbf a)\bigr|^{\gamma}\,d\xi\Bigr)^{\frac 1{\gamma}}.
\label{step1}
\end{align}
Note that $\h{f}$ is entire analytic (and so it is defined at $\xi-i\tau\mathbf a$)
because $f$ has compact support. Since
$\h{\nabla f}(\xi)=i\xi \h{f}(\xi),$
 from \eqref{hat-g-f} with $h(x) =(h_1(x),\,\ldots ,\, h_n(x))=e^{-\la \tau\mathbf{a},x\ra}\,\nabla f(x)$ we get
\[
\h{h}(\xi)=\h{\nabla f}(\xi-i\tau\mathbf a)=i(\xi-i\tau\mathbf a)\h{f}(\xi-i\tau\mathbf a).
\]
Hence
\begin{align*}
&\Bigl(\int_{\R^{n}}\bigl|(\xi-i\tau\mathbf a)\h{f}(\xi-i\tau\mathbf
a)\bigr|^{\gamma}\,d\xi\Bigr)^{\frac 1{\gamma} }\\
  {}={}& \Bigl(\int_{\R^{n}}|\h{h}(\xi)|^{\gamma}\,d\xi\Bigr)^{\frac 1{\gamma} }=
\Bigl(\int_{\R^{n}}\Bigl(\sum_{j=1}^{n}|\h{h}_{j}(\xi)|^{2}\Bigr)^{\frac \gamma
2}\,d\xi\Bigr)^{\frac 1{\gamma} }\\
  \le &\Bigl(\int_{\R^{n}}\Bigl(\sum_{j=1}^{n}|\h{h}_{j}(\xi)|\Bigr)^{\gamma}\,d\xi\Bigr)^{\frac 1{\gamma} }\le
\sum_{j=1}^{n}\Bigl(\int_{\R^{n}}|\h{h}_{j}(\xi)|^{\gamma}\,d\xi\Bigr)^{\frac 1{\gamma} },
\end{align*}
where the first inequality holds trivially and the second is Minkowski's
inequality.

Let us use Pitt's inequality with $p\le \gamma$:
\begin{equation}\label{pitt-1v}
\|\h{f}\|_\gamma\le C_{p,\gamma} A_{1,v} \|f\,v^{\frac 1p}\|_p,
\end{equation}
 where
\begin{equation}\label{cond-1-v}
A_{1,v}:=\sup_{s>0}\Bigl(\int_{0}^{\frac{1}{s}} dt\Bigr)^{\frac 1{\gamma} }
\Bigl(\int_{0}^{s}(1/v)^{*}(t)^{\frac{1}{p-1}}\,dt\Bigr)^{\frac 1{p'} }<\infty.
\end{equation}
As in Step 1, we apply Lemma \ref{lem-psi} with
$\beta_{1}=\beta_{2}=\frac {p'} \gamma$. We obtain $\frac {p'} \gamma\leq 1$ and
\[
A_{1,v}^{p'}\asymp \sup_{s>0}s^{-\frac {p'} \gamma}
\int_{0}^{s}(1/v)^{*}(t)^{\frac{1}{p-1}}\,dt\asymp
\sup_{s>0}s^{1-\frac {p'} \gamma}(1/v)^{*}(s)^{\frac{1}{p-1}}.
\]
It follows that
\[
A_{1,v}^{p}= A_{1,v}^{p'(p-1)}\asymp \sup_{s>0}s^{\frac
{p}{\gamma'}-1}(1/v)^{*}(s)=A_{v}^{p}<\infty.
\]
Applying \eqref{pitt-1v} with $f$ replaced by $h_{j}$, $j=1,\ldots,n$, we gather
\begin{align}
\Bigl(\int_{\R^{n}}|\h{h}_{j}(\xi)|^{\gamma}\,d\xi\Bigr)^{\frac 1{\gamma}
}&\lesssim A_{v}\Bigl(\int_{\R^{n}}|h_{j}(x)|^{p}v(x)\,dx\Bigr)^{\frac 1p}
\notag \\ &\lesssim
A_{v}\Bigl(\int_{\R^{n}}\Bigl(\sum_{k=1}^{n}|h_{k}(x)|^{2}\Bigr)^{\frac
p2}v(x)\,dx\Bigr)^{\frac 1p} \notag \\ \label{h-nabla-f}
&=A_{v}\Bigl(\int_{\R^{n}}|e^{-\la \tau\mathbf{a},x\ra}\nabla
f(x)|^{p}v(x)\,dx\Bigr)^{\frac 1p}.
\end{align}

This, together with \eqref{step1} proves part (a) of the theorem.

\smallbreak
(b) Let $1<\gamma\le q<p$. We proceed as in the proof of part (a) to obtain
\eqref{pitt-u-wt}, provided that \eqref{cond-u-wt} holds. We note that we
assume $\frac{n}{n-1}<\gamma$ when $\tau=0$.

Analogously, we get \eqref{h-nabla-f}, but instead of \eqref{cond-1-v}
we use \eqref{2e-weight-cond} with $u=1$, $w=v$ and
$\gamma<p$. Then we have
\begin{align*}
A_{1,v}^{r}&=\int_0^\infty s^{-\frac{r}{\gamma}}
\Bigl(\int_0^{s}(1/v)^*(t)^{\frac{1}{p-1}}\,dt\Bigr)^{\frac{r}{\gamma'}}
(1/v)^*(s)^{\frac{1}{p-1}}\,ds\\
&\asymp\int_0^\infty s^{-\frac{r}{\gamma}}\,\frac{d}{ds}
\Bigl(\int_0^{s}(1/v)^*(t)^{\frac{1}{p-1}}\,dt\Bigr)^{\frac{r}{\gamma'}+1}
 \,ds,
\end{align*}
where $\frac1r=\frac 1{\gamma} -\frac 1p$. After integrating by parts, we get
\[
A_{1,v}^{r}\asymp \int_0^\infty s^{-\frac{r}{\gamma}-1}
\Bigl(\int_0^{s}(1/v)^*(t)^{\frac{1}{p-1}}\,dt\Bigr)^{\frac{r}{p'}}\,ds=
\tilde{A}_{v}^{r}<\infty.
\]

This proves part (b) of the theorem.

\subsection{Corollaries and remarks}
Let us first discuss the conditions on $\gamma$ in Theorem
\ref{T1-grad-heinig}. We recall that in part (a) of Theorem
\ref{T1-grad-heinig} we assume $1<p\le q<\infty$ and $\max{}(p,p')\le \gamma\le
q$; when $\tau=0$ we assume also $\frac{1}{n}<\frac{1}{\gamma'}$.

Note that this extra assumption on $\gamma$ is not necessary when $n\ge 3$.
Indeed, from $\max{}(p,p')\le \gamma\le q$ follows that $2\leq \gamma\leq q$
and $q'\leq \gamma'\leq 2$; thus, $\frac{1}{n}<\frac{1}{\gamma'}$ whenever
$n\ge 3$.

When $n=1$, the inequality $\frac{1}{n}<\frac{1}{\gamma'}$ (or: $\gamma>\frac{n}{n-1}$) can never be satisfied by $\gamma'$ and only the case $\frac{1}{n}\ge \frac{1}{\gamma'}$ is possible. In fact, the condition $\max{}(p,p')\le \gamma\le q$ always implies
$\frac{1}{2}\le \frac{1}{\gamma'}<1$.

   %

If $n=2$ we can either have $\frac{1}{n}<\frac{1}{\gamma'}$ or $\frac{1}{n}\ge
\frac{1}{\gamma'}$. Note that $\frac{1}{2}\ge \frac{1}{\gamma'}$ implies that
$p=\gamma=2$.

\medskip

For applications, it is important to simplify the expression for
$A_{u}^{q}(1)$ in \eqref{u*-cond}. 
Recall that, when $\tau>0$, $A_{u}(\tau)\le \max{}(\tau^{-1},1)A_{u}(1)$ (see
Remark \ref{rem-ctau-bound}). We prove the following

\begin{Cor}\label{cor-Au1}
Let $1< p\leq q<\infty$ and let $\max{}(p,p')\le \gamma\le q$.

\textup{(i)} If $n\ge 2$ and $\frac{1}{n}<\frac{1}{\gamma'}$, then
$\frac{1}{\gamma'}\le \frac{1}{n}+\frac{1}{q}$ and
\begin{equation*}\label{e-sup-u}
A_{u}^{q}(1)\asymp \sup_{s>0}s^{(1-q(\frac{1}{\gamma'}-\frac{1}{n}),\,0)}u^*(s).
\end{equation*}

\textup{(ii)} If $n=2$ and $p=\gamma=2$, then
\[
A_{u}^{q}(1)=\sup_{s>0}\left(\ln{}(s^{-1}+1)\right)^{q/2} \int_{0}^s
u^*(t)\,dt.
\]

\textup{(iii)} If $n=1$, then 
\[
A_{u}^{q}(1)\asymp \sup_{s>0}s^{(0,-\frac{q}{\gamma'})}\int_{0}^s u^*(t)\,dt.
\]
\end{Cor}

\begin{proof}
(i) Recall that
\[
A_{u}^{q}(\tau)=
 \sup_{s>0} \int_{0}^s u^*(t)\,dt
\Bigl(\int_{0}^{\frac{1}{s}}
(t+\tau^{n})^{-\frac {\gamma'}{n} }\,dt\Bigr)^{\frac q{\gamma'}},\quad \tau>0.
\]
For $\tau=1$ and $\frac{1}{n}<\frac{1}{\gamma'}$, we have
\begin{align*}
\Bigl(\int_{0}^{\frac{1}{s}}
(t+1)^{-\frac {\gamma'}{n} }\,dt\Bigr)^{\frac{q}{\gamma'}}&\asymp
\bigl((s^{-1}+1)^{1-\frac {\gamma'}{n} }-1\bigr)^{\frac q {\gamma'} }\\
&\asymp s^{(-q(\frac{1}{\gamma'}-\frac{1}{n}),\,-\frac{q}{\gamma'})},\quad s>0.
\end{align*}
Hence
\[
A_{u}^{q}(1)\asymp
 \sup_{s>0} s^{(-q(\frac{1}{\gamma'}-\frac{1}{n}),\,-\frac{q}{\gamma'})}
 \int_{0}^s u^*(t)\,dt,
\]
where $q(\frac{1}{\gamma'}-\frac{1}{n})>0$ and $\frac{q}{\gamma'}\ge 1$, since
$\gamma'\le 2\le q$.

Now we can apply Lemma \ref{lem-psi} with
$\beta_{1}=q(\frac{1}{\gamma'}-\frac{1}{n})$, $\beta_{2}=\frac{q}{\gamma'}\ge
1$. We obtain $q(\frac{1}{\gamma'}-\frac{1}{n})\ge 1$ or $\frac{1}{\gamma'}\le
\frac{1}{n}+\frac{1}{q}$ and
\[
A_{u}^{q}(1)\asymp
\sup_{s>0}s^{(1-q(\frac{1}{\gamma'}-\frac{1}{n}),\,0)}u^*(s).
\]

Part (ii) is obvious. To prove part (iii), we note that $\gamma'>1$, which
gives $\int_{0}^{\frac{1}{s}} (t+1)^{-\gamma'}\,dt\asymp s^{(0,-1)}$.
\end{proof}




\begin{proof}[Proof of Corollary \ref{Cor-power-weight}] Recall that in this corollary
$
u(x)=|x|^{(-\alpha_{1},-\alpha_{2})}$, $v(x)=|x|^{(\beta_{1},\beta_{2})}$ with
$\alpha_j,\,\beta_j\ge 0$. We consider the case when $1<p\le q<\infty$ and
$\gamma\in [\max{}(p,p'),q]$, with $\frac{1}{n}<\frac{1}{\gamma'}\le
\frac{1}{n}+\frac{1}{q}$, and we let $\tau=0$ or $\tau=1$.

Since $w^{*}(s)\asymp w_{0}(s^{\frac 1n})$, $s>0$, for any non-increasing radial
weight function $w(x)=w_{0}(|x|)$ we have
\[
u^{*}(s)\asymp s^{(-\frac{\alpha_{1}}{n},-\frac{\alpha_{2}}{n})},\quad
(1/v)^{*}(s)\asymp s^{(-\frac{\beta_{1}}{n},-\frac{\beta_{2}}{n})}.
\]
whenever $\alpha_{j},\beta_{j}\ge 0$.

The expression \eqref{u*-cond}, Corollary \ref{cor-Au1} (i), and \eqref{v*-cond}
imply that for $s>0$
\[
u^{*}(s)\lesssim \begin{cases}
s^{q(\frac{1}{\gamma'}-\frac{1}{n})-1},& \tau=0,\\
s^{(q(\frac{1}{\gamma'}-\frac{1}{n})-1,0)},& \tau=1,
\end{cases}\qquad
(1/v)^{*}(s)\lesssim s^{1-\frac {p}{ \gamma'}}.
\]

It is easy to see that when $a_{j},b_{j}\ge 0$, the inequality $s^{(-a_{1},-a_{2})}\lesssim s^{(-b_{1},-b_{2})}$,
  holds if and only if $a_{1}\le b_{1}$, $a_{2}\ge b_{2}$. It follows
that
\[
\alpha_{1}\le n\Bigl(1-\frac{q}{\gamma'}+\frac{q}{n}\Bigr),\quad
\begin{cases}
\alpha_{2}\ge n\bigl(1-\frac{q}{\gamma'}+\frac{q}{n}\bigl),& \tau=0,\\
\alpha_{2}\ge 0,& \tau=1,
\end{cases}
\]
and
\[
0\le \beta_{1}\le n\Bigl(\frac{p}{\gamma'}-1\Bigr),\quad
\beta_{2}\ge n\Bigl(\frac{p}{\gamma'}-1\Bigr)
\]
which proves \eqref{cond-alpha} and \eqref{cond-beta}.

\medskip
 {To prove} \eqref{ab-n-cond} we use a standard homogeneity argument. Let
us consider \eqref{u-v-a-ineq} (which by Remark
\ref{rem-u-v-a-ineq} is equivalent to \eqref{e2-weighted-grad}) with $f=f_{\lambda}(x)=f(\lambda x)$ for some $f\in
C^\infty_0(\R^n)$ and $\lambda>0$. We obtain
\[
\||x|^{-\frac{\alpha}{q}}f_{\lambda}\|_q\le
c_\tau\||x|^{\frac{\beta}{p}}\,(\tau \mathbf{a}f_{\lambda}+\lambda (\nabla
f)_\lambda\|_p.
\]
After the change of variables $x\mapsto \lambda^{-1}x$,
we get
\begin{equation}\label{e-reduced}
\lambda^{\frac{\alpha}{q}-\frac{n}{q}+\frac{\beta}{p}+\frac{n}{p}-1}
\||x|^{-\frac{\alpha}{q}}f\|_q\le c_\tau
\||x|^{\frac{\beta}{p}}\,(\lambda^{-1}\tau \mathbf{a}f+\nabla f)\|_p.
\end{equation}
The limits of the two sides of the inequality \eqref{e-reduced}, as $\lambda\to 0$ or as $\lambda \to \infty$, must be the same. If $\tau=0$ the right-hand side of \eqref{e-reduced} does not depend on $\lambda$, so we must have $\frac{\alpha}{q}-\frac{n}{q}+\frac{\beta}{p}+\frac{n}{p}-1=0$.

If $\tau>0$, we must have
\[
\lambda^{\frac{\alpha}{q}-\frac{n}{q}+\frac{\beta}{p}+\frac{n}{p}-1}\lesssim
\begin{cases}
 \lambda^{-1}, & \lambda\to 0,\\ 1,& \lambda\to
\infty.
\end{cases}
\]
so necessarily $\frac{\alpha}{q}-\frac{n}{q}+\frac{\beta}{p}+\frac{n}{p}-1\leq 0$.

\end{proof}

\section{Uniqueness problems}\label{sec-unic-prob}

In this section and in Section~\ref{sec-lin-sys} we use the inequality
(\ref{e-weighted-grad-tau}) to prove uniqueness questions for solutions of
partial differential equations and systems.
 First, we state some definitions and preliminary results.

Let $\alpha=(\alpha_1, \, \dots,\, \alpha_n)$ be a vector with non-negative
integer components; we use the notation $|\alpha|=\alpha_1+\dots+\alpha_n$
and $\partial^\alpha_x f =
\frac{\partial^{\alpha_1}}{\partial^{\alpha_1}_{x_1}}\cdots
\frac{\partial^{\alpha_n}}{\partial^{\alpha_n}_{x_n}}f $.

Let $D\subset \R^n$ open and connected and let $1\leq p<\infty$. Recall that $W^{m, p}_0(D)$ is the closure of $C^\infty_0(D)$ with respect to the Sobolev norm $\|f\|_{W^{m,p}_0(D)}= \sum_{ |\alpha|=0}^m \| \partial^\alpha_x f\|_p$,
       When $m=1$, and $D $ is bounded in at least one direction, the classical Poincare' inequality states that $\|f\|_{L^p(D)}\leq C\|\nabla f\|_{L^p(D)}$
       (see e.g. \cite{Brezis}); thus, the Sobolev norm in $W^{1,p}_0(D)$ is equivalent to $ \|\nabla f\|_{L^p(D)}$.

Given the weight $v\colon D \to\ [0, \infty]$ and $1\leq p<\infty$, we let
$W^{m, p, v }_0(D)$ be the closure of $C^\infty_0(D)$ with respect to the norm
$\|f\|_{W^{1,p,v} _0(D)}= \sum_{|\alpha|=0}^m\|v^{\frac 1p}
\partial^\alpha_x f\|_{p}$. We use the standard notation $L^{p,v}(D)$ or $L^p(D
,\, v\,dx)$ for the closure of $C^\infty_0(D)$ with respect to the norm
$\|v^{\frac 1p} f\|_{p}$.

Let $P(\partial) =\sum_{|\alpha|=0 }^ma_\alpha \partial^\alpha_x $ be a linear
 partial differential operator of order $m>0$ with complex constant coefficients.
We let $P(-\partial)u= \sum_{|\alpha|=0 }^m
\overline{a_\alpha}\,(-1)^{|\alpha|}\partial^\alpha_x u$.

A {\it weak solution} (or: a {\it solution in distribution sense}) of the equation
$P(\partial)f=0$ on a domain $D\subset\R^n$ is a distribution $f\in W^{m,p}(D)$ that satisfies $\int
_{D} f(x)\, P(-\partial)\phi (x)\, dx=0 $ for every $\phi\in
C^\infty_0(D)$.
Weak solutions for non linear partial differential operators can be defined on
a case-by-case basis. See e.g. \cite{E} or other standard textbooks on partial
differential equations for details. We will often consider differential
inequalities in the form of $|P(\partial) f| \leq |V f|$ on a given domain $D$;
by that we mean that the inequality $|P(\partial) f(x)| \leq |V f(x)|$ is
satisfied a.e. in $D$, i.e., it is satisfied pointwise with the possible
exception of a set of measure zero.

\medskip

\subsection{Unique continuation and Carleman method}
Let $P(\partial)$ be a homogeneous partial differential operator of order $m\ge
1$. Clearly, $f\equiv 0$ is a solution of the equation $P(\partial)f=0$ on any
domain $D\subset \R^n$. It is natural to ask whether this equation has also
nontrivial solutions, i.e., distributions in some suitable Sobolev space that
satisfy the equation in distribution sense and are not identically $=0$. In
particular it is natural to ask whether (1), (2) or (3) below are satisfied or not
on a given domain $D$.


 \begin{enumerate}


 \item \textit{Uniqueness for the Dirichlet problem.} The only solution of the Dirichlet problem
 $\begin{cases} P(\partial)f=0, \\ f\in W_0^{m,p}(D) \end{cases}\hskip-1em$ \quad is $ f\equiv 0$.

 \item \textit{Weak unique continuation property }
 (or: unique continuation from an open set). Every solution of the equation $P(\partial)f=0$ which is $ \equiv 0$ on an open subset of $D$ is $ \equiv 0$.

\item \textit{Strong continuation property } (or: unique continuation from a point). Let $x_0\in D$. Every solution of the equation $P(\partial)f=0$ that satisfies
\[
\lim_{r\to 0} r^{-N} \int_{|x- x_0|<r} |f(x)|^2\,dx =0
\]
for every $N>0$ is $ \equiv 0$.

  \end{enumerate}

For other relevant unique continuation problems see the survey paper~\cite{T}.

 Historically, the study of unique continuation originated from the uniqueness for the Cauchy problem;
 an equally strong motivation arose from some
fundamental questions in mathematical physics, with the study of the eigenvalues of the time-independent Schr\"odinger operator
$
H = -\Delta +V \
$
as a notable example. See \cite{S} and \cite{Simon} and also \cite{KT} and the references cited there.
 %

In 1939 T. Carleman used in \cite{C} a new weighted Sobolev inequality to show that the Schr\"odinger operator $H=-\Delta+V$ has the strong unique continuation property when $n=2$ and $V$ is bounded. Carleman's original idea has permeated the large majority of results on unique continuation. The weighted Sobolev inequality that he used in his proof has been widely generalized and applied to a vast array of problems in unique continuation and control theory.

A {\it Carleman-type inequality} for a differential operator $P(\partial)$ is a weighted
inequality of the form of
\begin{equation}\label{Carleman estimates}
\| \eta ^{\tau_k} f\|_{q} \leq
C\|\eta ^{\tau_k} P(\partial)f\|_{p},
\qquad f\in C^{\infty}_0(D),
\end{equation}
where $\eta\colon D\to [0,1)$, the sequence $\{\tau_k\}_{k\in\N} \subset (0,
\infty) $ increases to $+\infty$, the constant $C$ is independent of  the sequence of the $\tau_k$
and of $f$, and $1\leq p\leq q <\infty$. If \eqref{Carleman estimates} holds with a
suitable function $\eta$, a version of the argument used in the proof of
Theorem \ref{T1-UC-grad} can be applied to show that the operator $Q(\partial)
= P(\partial) -V(x) $ has the unique continuation property (2) or (3) (or some
variation of these properties) whenever $V\in L^{\frac{pq}{q-p}}(D)$.

The
literature on Carleman inequalities and unique continuation is very extensive.
A sample of references on unique continuation problems for second order
elliptic operators include the important \cite{J, JK, KRS, So} and the survey
papers \cite{W, KT2, T}.

The inequality \eqref{e2-weighted-grad} in Theorem \ref{T1-grad-heinig} can be
viewed as a weighted Carleman-type inequality for the operator
$P(\partial)f=|\nabla f|$. To the best of our knowledge, the inequality
\eqref{e2-weighted-grad} is new in the literature, even when $u(x)\asymp
v(x)\asymp 1$.

\subsection{Proof of Theorem \ref{T1-UC-grad}}

In this section we prove Theorem \ref{T1-UC-grad} and some corollary.

\begin{proof}[Proof of Theorem \ref{T1-UC-grad}]
Assume for simplicity that $f\equiv 0$ when $x_n<0$ (the proof is similar in
the general case). It is enough to show that $f\equiv 0$ also on the strip
$S_\epsilon=\{x\colon 0<x_n<\epsilon\}$, where $\epsilon>0$ will be determined
during the proof. Using Theorem \ref{T1-grad-heinig}~(a) with $\mathbf{a} =
(0,\ldots,0,1)$, $\tau\ge 1$ and   $c_\tau\le c_1$ (see
Remark~\ref{rem-ctau-bound}), the differential inequality \eqref{ineq-V-nable}
and H\"older's inequality with $\frac 1p=\frac 1q+\frac 1r$, we can write the
following chain of inequalities:
\begin{align*}
\|e^{-\tau x_n} fu^{\frac 1q}\|_{L^q(S^\epsilon)} & \leq c_1  \|e^{-\tau x_n}\nabla fv^{\frac 1p}\|_{L^p(\R^n)}
\\ & \leq c_1 \|e^{-\tau x_n}\nabla fv^{\frac 1p}\|_{L^p(S_\epsilon)} + c_1\|e^{-\tau x_n}\nabla fv^{\frac 1p}\|_{L^p(\{x_n>\epsilon\})}
\\
& \leq c_1 \|e^{-\tau x_n} f V \,v^{\frac 1p}\|_{L^p(S_\epsilon)} + c_1e^{-\tau \epsilon }\|\nabla fv^{\frac 1p}\|_{L^p(\{x_n>\epsilon\})}
\\
& \leq c_1\| V \,v^{\frac 1p} u^{-\frac 1q}\|_{L^r(S_\epsilon\cap \supp
f)}\|e^{-\tau x_n} fu^{\frac 1q}\|_{L^q(S^\epsilon)} + C'e^{-\tau \epsilon }.
\end{align*}
Here,  $\frac 1r=\frac 1p-\frac 1q$ and we have let $C'=c_1\|\nabla
fv^{\frac 1p}\|_{L^p(\{x_n>\epsilon\})}$. Note that    $C'$ does not depend on
$\tau$.

Since $V\in L^r (\supp f, \, v^{\frac rp} u^{-\frac rq}\,dx)$ we
can chose $\epsilon>0$ so that $c_1\|V \,v^{\frac 1p} u^{-\frac
1q}\|_{L^r(S_\epsilon\cap \supp f)} <\frac 12$. From the chain of inequalities
above, follows that
$$
\|e^{-\tau x_n} fu^{\frac 1q}\|_{L^q(S^\epsilon)}\leq \frac 12\,\|e^{-\tau x_n}
fu^{\frac 1q}\|_{L^q(S^\epsilon)}+ C'e^{-\tau \epsilon }.
$$
We gather
$$
\frac 12\,\|e^{ \tau (\epsilon-x_n)} fu^{\frac 1q}\|_{L^q(S^\epsilon)} \leq C'.
$$
Since $\epsilon-x_n>0$ on $S^\epsilon$, if $f\not\equiv 0$ the left-hand side of this inequality goes to infinity when $\tau$ goes to infinity; this is a contradiction because $C'$ does not depend on $\tau$ and so necessarily $f\equiv 0$ in $S_\epsilon$.
\end{proof}

%
%

\begin{Cor}\label{C-UC-power-weights}
Let $p$, $q$ and $\gamma$ be as in Theorem~\textup{\ref{T1-grad-heinig}~(a)}.
Let $u=|x|^{(-\alpha_1, -\alpha_2)}$ and $v= |x|^{(\beta_1,\, \beta_2)}$, with
$0\le \alpha_{1}\le n\bigl(1-\frac{q}{\gamma'}+\frac{q}{n}\bigr)$,
$\alpha_{2}\ge 0$ and $0\le \beta_{1}\le n\bigl(\frac{p}{\gamma'}-1\bigr)$,
$\beta_{2}\ge n\bigl(\frac{p}{\gamma'}-1\bigr)$.
Let $V=|x|^{(s_1, s_2)}$, with
\begin{equation}\label{e-1}
s_1>-\frac{n}{r}-\frac{\alpha_{1}}{q}-\frac{\beta_{1}}{p},
\end{equation}
and, if $\supp f$ is unbounded,
\begin{equation}\label{e-2}
s_2< -\frac {\alpha_2}{q}-\frac{\beta_2}{p} -\frac nr.
\end{equation}
Then, every solution 
of the differential inequality
$
|\nabla f|\leq V |f| $ is $\equiv 0$.
\end{Cor}
\begin{proof}
The weights $u$ and $v$ are as in Corollary \ref{Cor-power-weight}, so the
inequality \eqref{e2-weighted-grad} holds with $\tau>0$. By
Theorem~\ref{T1-UC-grad}, every solution 
of the differential inequality $ |\nabla f|\leq V |f| $ is $\equiv 0$ whenever
$Vv^{\frac 1p}u^{-\frac 1q}\in L^r(\supp f)$. We can see at once that
$Vv^{\frac 1p}u^{-\frac 1q}=|x|^{(t_1, t_2)}\in L^r(\supp f)$ if and only if
$t_1= s_1+\frac {\alpha_1}{q}+\frac{\beta_1}{p}>-\frac nr$ and, if $\supp f$ is
unbounded, $t_2= s_2+\frac {\alpha_2}{q}+\frac{\beta_2}{p}<-\frac nr$, which is
equivalent to \eqref{e-1} and \eqref{e-2}. This concludes the proof.
\end{proof}

\begin{Rem}\label{rem-Vepsilon}
From the inequalities above and the assumptions on $\alpha_j$, $\beta_j$, and
$\gamma'$ (see Corollary~\ref{Cor-power-weight}) follows that
\begin{align*}
t_1 &\leq s_1+\frac nq\Bigl(1-\frac{q}{\gamma'}+\frac qn\Bigr)+\frac
np\Bigl(\frac p{\gamma'}-1\Bigr)=s_1-\frac nr +1.
\\
t_2&\ge s_2+\frac np\Bigl(\frac p{\gamma'}-1\Bigr)= s_2+\frac{n}{\gamma'}-\frac
np> s_2-\frac np+1.
\end{align*}
The condition $t_1>-\frac nr$ yields $s_1> -1$. We can see at once that
$t_2<-\frac nr$ yields $s_2<\frac nq-1$. In particular, $V= |x|^{-1+\epsilon}$
with $0<\epsilon <\frac nq$, satisfies the assumptions of Corollary
\ref{C-UC-power-weights}. If $f$ has compact support, then we can
omit the condition on $t_2$ and assume only $\epsilon>0$.

Potentials $V(x)= C|x|^{-s}$, with $s,\,C>0$ are known as \textit{Hardy
potentials} in the literature. They appear in the relativistic Schr\"odinger
equations and in problem of stability of relativistic matter in magnetic
fields. See e.g. \cite{He} and the introduction to \cite{F} and \cite{FF}, just
to cite a few.

It is proved in \cite{DO} that when $\mathcal{L}$ is the Dirac operator in
dimension $n\ge 2$ (see Section~\ref{ssec-dirac}) the differential inequality
$|\mathcal{L}f| \leq C |x|^{-1} |f|$ has the strong unique continuation
property from the point $x_0=0$ whenever $C\leq 1$. We conjecture that also the
differential inequalities $|\nabla f|\leq C |x|^{-1} |f|$
has the strong unique continuation
property from  the origin  when $C$ is sufficiently small.
\end{Rem}

\subsection{Proof of Theorem~\ref{T-UC-dir1}}

Recall that the solution $f$ of the Dirichlet problem \eqref{e-dir1} is
intended
in distribution sense, i.e., $f$ satisfies
\begin{equation}\label{e-pde1}
\int_D \la\nabla \psi,\nabla f\ra|\nabla f|^{p-2}\, v\,dx= \int_D \psi\,Vf |f|^{p-2} \, v\, dx
\end{equation} for every $\psi\in C^\infty_0(D)$.
To prove Theorem \ref{T-UC-dir1} we need two important lemmas:
\begin{Lemma}\label{L-4.1} Suppose that the weighted gradient inequality
\begin{equation}\label{e1-weighted-grad1}
\|u^{\frac 1q} f\|_q\leq c_0\|v^{\frac 1p} \nabla f\|_p, \quad f\in
C^\infty_0(D)
\end{equation}
holds with exponents $1\leq p,\,q <\infty$. Then the space $W^{1,p,v}_0(D)$
embeds into $L^q(D , u\,dx)$ and $ \|f\|_{L^q(D , \, u\,dx)}\leq c_0\| \nabla
f\|_{L^p(D ,\, v\,dx)}$.
\end{Lemma}
\begin{proof}
Fix $f\in W^{1,p,v}_0 (D)$; let $\{f_n\}_{n\in\N}\subset C^\infty_0(D)$ be a
sequence that converges to $f$ in the Sobolev norm
$\|{\,\cdot\,}\|_{W^{1,p,v}_0 (D)}$. Thus, $\{f_n\}$ is a Cauchy sequence in
$W^{1,p,v}_0 (D)$; for every $\epsilon>0$ we can chose $N>0$ such that
$$
 \| f_n-f_m\|_{W^{1,p,v}(D)} = \| v^{\frac 1p}(f_n-f_m)\|_{L^p(D)}+\|v^{\frac 1p} \nabla( f_n-f_m )\|_{L^p(D)}<\epsilon
$$
whenever $n,\,m>N$; thus, $\|v^{\frac 1p} \nabla( f_n-f_m
)\|_{L^p(D)}<\epsilon$. By \eqref{e1-weighted-grad1},
$$
\|u^{\frac 1q} ( f_n-f_m )\|_{L^q (D)}\leq c_0\|v^{\frac 1p} \nabla( f_n-f_m )\|_{L^p(D)}<c_0\epsilon.
$$
We have proved that $\{f_n\}$ is a Cauchy sequence in $L^q(D , \, u\,dx)$ (which is complete) and so it converges to $f$ also in $ L^q(D , \, u\,dx)$.
 We gather
\begin{align*}
 \|f \|_{L^q(D , \, u\,dx)}&=\lim_{n\to\infty} \|f_n \|_{L^q(D , \, u\,dx)}\leq
 c_0 \lim_{n\to\infty} \|\nabla f_n \|_{L^p(D , \, v\, dx)}\\
 &= c_0\|\nabla f \|_{L^p(D , \, v\, dx)}
\end{align*}
 as required.
 \end{proof}

\begin{Lemma}\label{L-GreenG} Suppose that the weighted gradient inequality
\eqref{e1-weighted-grad1} holds with $1<p<q$. Let $f$ be a solution to the
Dirichlet problem \eqref{e-dir1}, with $|V|^{\frac 1p}\in L^r (D, \, v^{\frac
rp} u^{-\frac rq}\,dx)$. We have
\begin{equation*}\label{greenh2}
\int_D |\nabla f|^p v\,dx= \int_D V |f|^p v\,dx.
\end{equation*}
\end{Lemma}

\begin{proof}
Let $\{\psi_n\}$ be a sequence of functions in $C^\infty_0(D)$ that converges
to $\overline f$, the complex conjugate of $f$, in $W^{1,p,v}_0(D)$. We show
first that $\lim_{n\to\infty} \int_D \la \nabla \psi_n,\nabla f\ra |\nabla
f|^{p-2} v\,dx= \int_D |\nabla f|^p v\, dx$. Indeed,
\begin{align*}
 &\int_D \bigl(\la \nabla \psi_n,\nabla f\ra |\nabla f|^{p-2} - |\nabla f|^p\bigr)
 v\, dx \\
 {}={}&
  \int_D \bigl(\la \nabla \psi_n,\nabla f\ra |\nabla f|^{p-2} -
  \la \nabla \overline f,\nabla f\ra |\nabla f|^{p-2}\bigr) v\, dx
  \\
 {}={}&
  \int_D \la \nabla \psi_n-\nabla \overline f,\,\nabla f |\nabla f|^{p-2}\ra\,v\, dx
  \\
  \leq & \| ( \nabla \psi_n - \nabla \overline f)v^{\frac 1p}\|_p\,\||\nabla f|^{p-1}\, v^{\frac 1{p'}}\|_{p'} \\
  {}={}&
 \| \nabla (\psi_n - \overline f)v^{\frac 1p}\|_p \,\|\, |\nabla f| \, v^{\frac 1p} \|_{p }^{\frac p{p'}}
\end{align*}
and $\lim_{n\to\infty} \| \nabla (\psi_n - \overline f)v^{\frac 1p}\|_p=0$, as required.

In view of \eqref{e-pde1}, we have that $$\int_D \la \nabla \psi_n,\nabla
f\ra |\nabla f|^{p-2} v\,dx= \int_D \psi_n\,Vf |f|^{p-2} \, v\, dx;$$ to
complete the proof it suffices to show that $$\lim_{n\to\infty} \int_D
\psi_n\,Vf |f|^{p-2} \, v\, dx = \int_D \,V |f|^{p } \, v\, dx$$ when
$|V|^{\frac 1p}\in L^r(D,\, v^{\frac rp}u^{-\frac rq}dx)$. By Lemma
\ref{L-4.1}, $\psi_n$ converges to $\bar f$ in $L^{ q} (D,\, u\,dx)$.
Using Holder's inequality with $ \frac pr+\frac pq =1 $, we gather
\begin{align*}
&\int_D \bigl(\psi_n V f |f|^{p-2} - V |f|^p\bigr)\,v\, dx\leq \int_D |V v
u^{-\frac pq}|\,| f|^{p-1} \, |\psi_n - \bar f | u^{\frac pq}\, \, dx\\
{}\leq{}&\Bigl(\int_D |v V u^{-\frac pq} |^{\frac rp}\,dx\Bigr)^{\frac
pr}\Bigl(\int_D |f|^{(p-1)\frac qp} |\psi_n-\bar f|^{\frac qp}
u\,dx\Bigr)^{\frac pq }\\ {}={}&\Bigl(\int_D (|V|^{\frac 1p} v^{\frac 1p} u^{-\frac
1q} )^{r}\,dx\Bigr)^{\frac pr}\Bigl(\int_D |f|^{ \frac q{p'}} |\psi_n-\bar
f|^{\frac qp} u\,dx\Bigr)^{\frac pq }.
\end{align*}
We let $C= \bigl(\int_D (|V|^{\frac 1p} v^{\frac 1p} u^{-\frac 1q}
)^{r}\,dx\bigr)^{\frac pr}$ and we apply H\"older's inequality (with $\frac
1p+\frac 1{p'}=1$) to the remaining integral. We obtain
\begin{align}
&\int_D \bigl(V f |f|^{p-2}\psi_n - V |f|^p\bigr)\,v\, dx \notag\\
{}\leq{}&C \Bigl(\int_D
|f|^{q}u\, dx\Bigr)^{\frac p{qp'} } \Bigl(\int_D | \psi_n-\bar f|^{q}u\,
dx\Bigr)^{\frac 1{q } } \notag\\
{}={}&C \|f u^{\frac 1q}\|_{q}^{p-1} \|(\psi_n-\bar f ) u^{\frac 1q}\|_{q }.
\label{e-1n}
\end{align}
By assumption, $\lim_{n\to\infty} \|(\psi_n-\bar f) u^{\frac 1q}\|_{q }=0$; by
Lemma \ref {L-4.1}, $\|f u^{\frac 1q}\|_{q} <\infty$, and so the right-hand
side of \eqref{e-1n} goes to zero when $n\to\infty$ as required.
\end{proof}

\medskip
\begin{proof}[Proof of Theorem \ref{T-UC-dir1}] Since the weights $u$ and $v$ are as in Theorem \ref{T1-grad-heinig}, the weighted gradient inequality \eqref{e1-weighted-grad1} holds.
By Lemma \ref{L-GreenG}
and H\"older's inequality (with $ \frac pq+\frac pr=1$) we have the following chain of inequalities
\begin{align*}
\| f u^{\frac 1q}\|_{L^q(D)}^p & \leq c_0^p\|\nabla f\, v^{\frac 1p}\|_{L^p(D)}^p= c_0^p \int_{D } v| \nabla f|^{p}dx
\\
  &=
  c_0^p \int_{D }V v|f|^{p} \ dx
\leq
    c_0 ^p \int_{D }V_+ v u^{-\frac pq}\, |f|^{p} u^{ \frac pq}\ dx
\\
  &\leq
    c_0^p\Bigl(\int_{D }V_+^{\frac rp} v^{\frac rp} u^{-\frac rq}\, dx\Bigr)^{\frac pr}
    \Bigl(\int_D |f|^{q} u \, dx\Bigr)^{\frac pq}\\
    &\leq
    c_0^p\|V_+^{\frac 1p}\|_{L^r(D, \, v^{\frac rp} u^{-\frac rq}\, dx)}^p \| f u^{\frac 1q}\|_{L^q(D)}^p.
\end{align*}
We obtain $\| f u^{\frac 1q}\|_{L^q(D)}\bigr(1- c_0^p\|V_+^{\frac
1p}\|_{L^r(D,\, v^{\frac rp} u^{-\frac rq}\, dx)}\bigr)\leq 0 $; this inequality
is possible only if either $ c_0^p\,\|V_+^{\frac 1p}\|_{L^r(D,\, v^{\frac rp}
u^{-\frac rq}\, dx)} \ge 1$ or $f \equiv 0$ in $D$.
\end{proof}

\bigskip

\section{Linear systems of PDE and the Dirac operator}\label{sec-lin-sys}



We use the following notation: If $\vec p= (p_1, \ldots, p_m) \in \R^m$, we let
$|\vec p\,|= (p_1^2+\ldots+p_m^2)^{\frac 12}$. If $\mathbf{A}$ is a matrix with
rows $A_1, ,\ldots, A_N$, we will let $|\mathbf{A}|=
(|A_1|^2+\ldots+|A_N|^2)^{\frac 12}$. Note that, by Cauchy Schwartz inequality,
$$|\mathbf{A}\vec p\,|=(\la A_1,\, \vec p\,\ra^2+\ldots+ \la A_N,\, \vec
p\,\ra^2)^{\frac 12}\leq (|A_1|^2+\ldots+|A_N|^2)^{\frac 12} |\vec p\,|=|\mathbf{A}|\, |\vec p\,|.
$$

Let $\vec F =(f_1, \ldots, f_N)\in C^\infty_0(\R^n,\, \R^N)$. We denote with
$\nabla \vec F $ the $N\times n$ matrix whose rows are $\nabla f_1, \ldots,
\nabla f_N$.

Unless otherwise specified, we assume that $p$, $q$, $u$ and $v$ are as in
Theorem \ref{T1-grad-heinig}~(a) and that $\frac 1r=\frac 1p-\frac 1q$.

\medskip In this section we use the Carleman inequality
\eqref{e-weighted-grad-tau} to prove unique continuation properties of systems
of linear partial differential equations of the first order.

\subsection{Linear systems of PDE}
   Most of the first order systems considered in the literature are in the form of
  \begin{equation}\label{lin-1}
  \sum_{j=1}^n \mathbf{L}_j (x) \partial_{x_j}\vec F = V(x) \vec F,
  \end{equation}
where $\vec F =(f_1, \ldots, f_N) $ and the $\mathbf{L}_j(x)$ and $V$ are
$M\times N$ matrices defined in a domain $D\subset \R^n$. We let
$\mathbf{L}(x)(\vec F)= \sum_{j=1}^n \mathbf{L}_j (x) \partial_{x_j}\vec F $.
Differential inequalities in the form of
\begin{equation}\label{e5-diffineq- sys-2}
  |\mathbf{L}(x)\vec F|\leq |\mathbf{V}(x) \vec F|
\end{equation}
are also considered. In some of early papers on the subject, it is proved that
solutions of elliptic systems in the form of \eqref{lin-1} that vanish of
sufficiently high order at the origin are $\equiv 0$; see \cite{Cosner,DNi,R}
and the references cited in these papers for definitions of elliptic systems. A
classical method of proof is to reduce the systems to (quasi-) diagonal form;
this approach requires conditions on the regularity and the multiplicity of the
eigenvalues of the system that are often difficult to check; see \cite{Douglis,
Haya, HP, Uryu}. The strong continuation properties of systems of complex
analytic vector fields in the form of $\vec L u=0$ defined on a real-analytic
manifold is proved in \cite{BCP}.

We have found only a few papers in the literature where the Carleman method is
used to prove unique continuation properties of first-order systems. The
Carleman method often allows to prove unique continuation results for the
differential inequality \eqref{e5-diffineq- sys-2}, often with a singular
potential~$V$. In \cite[Theorem 4.1]{DN} Carleman estimates are used to prove
that \eqref{e5-diffineq- sys-2} has the weak unique continuation property when
$\vec L$ is a system of vector fields on a pseudoconcave Cauchy-Riemann (CR)
with some specified conditions and $V$ is bounded. In \cite{O} and \cite{O2} T.
Okaji considers systems in two independent variables, Maxwell's equations, and
the Dirac operator; he proved that the differential inequalities
\eqref{e5-diffineq- sys-2} with $|V(x)|\asymp |x|^{-1}$ has the strong unique
continuation property using sophisticated $L^2\to L^2$ Carleman estimates. See
also \cite{Tamura}, which improves results in \cite{O}.

\medskip

We prove the following
  \begin{Thm}\label{T- UC-ellypt}
  Let $\vec F \in W_0^{1,p, v}(\R^n,\, \R^N )$ be a
  solution of the differential inequality \eqref{e5-diffineq- sys-2}. Assume
  that $\vec F$ satisfies also
  \begin{equation}\label{Cond-ellipt}
  |\nabla \vec F|\lesssim |\mathbf{L}(x)\vec F|.
  \end{equation}
{If $|\mathbf{V}|\in L^{r}(\supp \vec F, \, u^{-\frac rq}v^{\frac rp}\,dx)$,
with $\frac 1p=\frac 1r+\frac 1q$ and $\vec F$ vanishes on one side of a
hyperplane,
then $\vec F\equiv 0$.}
\end{Thm}

In particular, for power weights $u$, $v$ as in
Remark~\textup{\ref{rem-Vepsilon}}, the differential inequality
\eqref{e5-diffineq- sys-2} does not have solutions with compact support support
that satisfy also \eqref{Cond-ellipt} if $V\asymp |x|^{- 1+\epsilon}$ for some
$\epsilon>0$.

\medbreak

Our unique continuation result is weaker than other results in the literature,
but it applies to first-order systems of linear partial differential equations
that satisfy only the assumptions \eqref{Cond-ellipt}. Furthermore, we consider
solutions in weighted Sobolev spaces and potential in weighted $L^r$ spaces
that, to the best of our knowledge, have not been considered in other papers.

  \medskip
Before proving Theorem \ref{T- UC-ellypt} we prove the following Lemma, which is an easy consequence of Theorem \ref{T1-grad-heinig}.
\begin{Lemma}\label{L-vect-valued}
Let ${\mathbf A}$ be a $N\times N$ invertible matrix. Under the assumptions of
Theorem~\textup{\ref{T1-grad-heinig}~(a)}, the following inequality holds for
all $\vec F\in C^\infty_0(\R^n,\, \R^N)$ and $\tau\ge 0$
\begin{equation}\label{e5-weighted-grad}
\|e^{-\tau \ell(x) }u^{\frac 1 q} \vec F \, \|_q\le c_{\tau,N,\mathbf
A}\|e^{-\tau \ell(x)}v^{\frac 1p}\, \mathbf A\nabla \vec F\|_p,
\end{equation}
where $c_{\tau,N,\mathbf A}=NC_{\mathbf A}c_\tau$ and $c_{\tau}$ is
the constant in \eqref{e2-weighted-grad}.
\end{Lemma}

\begin{proof}
Using Theorem \ref{T1-grad-heinig}~(a), the elementary
inequalities
\[
|\vec F|=(f_1^2+\ldots+f_N^2)^{\frac 12}\le |f_1|+\ldots+|f_N|,\quad |f_{j}|\le
|\vec F|,
\]
and Minkowsky's inequality, we obtain
\begin{align*}
\|e^{-\tau\ell(x) }u^{\frac 1 q} \vec F \, \|_q &\leq \sum_{j=1}^N \|e^{-\tau \ell(x)
}u^{\frac 1 q} f_j \, \|_q
\leq c_\tau\sum_{j=1}^N \|e^{-\tau \ell(x) }v^{\frac 1p} \nabla f_j \, \|_p
\\
&\leq c_\tau N \|e^{-\tau \ell(x) }v^{\frac 1p} \nabla \vec F \,
\|_p.
\end{align*}

If $\mathbf A$ is invertible, then, for every $\xi\in\R^n$, we have that
$|\mathbf A \vec \xi|\ge C_{\mathbf A}^{-1}|\xi| $ for some $C_{\mathbf A}>0$;
thus,
\[
\|e^{-\tau \ell(x) }u^{\frac 1 q} \vec F \, \|_q \leq c_\tau N C_{\mathbf
A}\|e^{-\tau \ell(x) }v^{\frac 1p}{\mathbf A} \nabla \vec F \, \|_p
\]
as required.
\end{proof}

\begin{proof}[Proof of Theorem \ref{T- UC-ellypt}]
We argue as in the proof of Theorem~\ref{T1-UC-grad}. Without loss
of generality, we can assume that $\vec F\equiv 0$ when $x_n<0$ and
$\mathbf{A}=\mathbf{I}$, where $\mathbf I$ is the $N\times N$ identity matrix.
For simplicity of notation, we  denote  with $c_{1}$ the constant $c_{1,N,\,\mathbf{I}}$
in Lemma \ref{L-vect-valued}. We show that $\vec F\equiv 0$ also on the strip $S_\epsilon= \{ x\colon
0<x_n<\epsilon\}$, for some $\epsilon>0$ to be determined during the proof.

Using \eqref{e5-weighted-grad} with $\ell(x)= x_n$ and  $\tau\ge 1$,
the differential inequality \eqref{Cond-ellipt}, H\"older's inequality and Remark~\ref{rem-ctau-bound}, we obtain
\begin{align*}
&\|e^{-\tau x_n} \vec F u^{\frac 1q}\|_{L^q(S^\epsilon)}\\
{}\leq{}&c_{1}\|e^{-\tau x_n}\nabla \vec F\,v^{\frac 1p}\|_{L^p(\R^n)}\\
{}\leq{}&c_{1}\|e^{-\tau x_n}\nabla \vec F\, v^{\frac 1p}\|_{L^p(S_\epsilon)} +
c_{1}\|e^{-\tau x_n}\nabla \vec F\, v^{\frac 1p}\|_{L^p(\{x_n>\epsilon\})}\\
{}\leq{}&c_{1}C\|e^{-\tau x_n} \mathbf{L}(x)(\nabla \vec F) \,v^{\frac
1p}\|_{L^p(S_\epsilon)} + c_{1}e^{-\tau \epsilon }\|\nabla \vec F v^{\frac
1p}\|_{L^p(\{x_n>\epsilon\})}\\
{}\leq{}&c_{1}C\|e^{-\tau x_n} \mathbf{V} \vec F \,v^{\frac
1p}\|_{L^p(S_\epsilon)} + c_{1}e^{-\tau \epsilon }\|\nabla \vec F v^{\frac
1p}\|_{L^p(\{x_n>\epsilon\})}\\
{}\leq{}&c_{1}C\||\mathbf{V} |\, v^{\frac 1p} u^{-\frac
1q}\|_{L^r(S_\epsilon\cap \supp \vec F)}\|e^{-\tau x_n} \vec F u^{\frac
1q}\|_{L^q(S^\epsilon)} + C'e^{-\tau \epsilon },
\end{align*}
where we have let $C'=c_{1}\|\nabla \vec F\, v^{\frac
1p}\|_{L^p(\{x_n>\epsilon\})}$.

{Since $|\mathbf{V}|\in L^{r}(\supp \vec F, \, u^{-\frac rq}v^{\frac rp}\,dx)$
we can chose $\epsilon>0$ so that} $c_{1}C\||\mathbf{V} | \,v^{\frac 1p}
u^{-\frac 1q}\|_{L^r(S_\epsilon\cap \supp \vec F)} <\frac 12$. We have obtained
$$
\|e^{-\tau x_n} \vec F u^{\frac 1q}\|_{L^q(S^\epsilon)}\leq \frac 12\,\|e^{-\tau
x_n} \vec F u^{\frac 1q}\|_{L^q(S^\epsilon)}+ C'e^{-\tau \epsilon }.
$$
In view of  $\epsilon-x_n>0$ on $S^\epsilon$, the left-hand side of this inequality
goes to infinity with $\tau$ unless $\vec F \equiv 0$ on $S_\epsilon$; this is
a contradiction because $C'$ does not depend on $\tau$, and so $\vec F \equiv
0$ in $S_\epsilon$.
\end{proof}

\medskip

Let $\mathbf{G}_1(x), \dots , \mathbf{G}_n(x)$ be $N\times n$ matrices defined
on a domain $D\subset\R^n$. We consider the operator $$\mathbf{G}(\vec F
)=\mathbf{G}(f_1, \ldots, f_N)= \sum_{j=1}^N \mathbf{G}_j (x) f_j$$ with $
f_j\in C^\infty_0(D)$.

In \cite{LNP}, systems in the form of $\nabla F= \mathbf{G} \vec F$ are
considered. These systems can be used to model linear elasticity (in
curvilinear coordinates) of linearly elastic shells. See \cite{CM} and the
references cited there. We prove the following
\begin{Thm}\label{T-Uc-Sys-1}
     Let $\vec F \in W_0^{1,p, v}(D,\, \R^N )$ be a
  solution of the differential inequality
\begin{equation}\label{e5-system-1}
|\nabla \vec F|\lesssim |\G \vec F |.
\end{equation}
If $|\G|\in L^{r}(\supp \vec F, \, u^{-\frac rq}v^{\frac rp}\,dx
)$,  with $\frac 1p=\frac 1r+\frac 1q$, and $\vec F$
vanishes on one side of a hyperplane,
%
then $\vec F\equiv 0$.
\end{Thm}
\begin{proof}
Assume for simplicity
that $\vec F\equiv 0$ when $x_n<0$ (the proof is similar in the general case).
We show that $\vec F\equiv 0$ also on the strip $0<x_n<\epsilon$, for some
$\epsilon>0$ to be determined during the proof.   As  in the proof of  Theorem \ref{T- UC-ellypt}, we use
\eqref{e5-weighted-grad} with $\mathbf{A}=\mathbf{I}$,  $\ell(x)= x_n$ and $\tau\ge 1$.
 For each $j=1,\ldots,N$, we use the differential
inequality \eqref{e5-system-1} and H\"older's inequality in the
following chain of inequalities
\begin{align*}
&\|e^{-\tau x_n} \vec F u^{\frac 1q}\|_{L^q(S^\epsilon)}\\
{}\leq{}&c_{1}\|e^{-\tau x_n}\nabla \vec F\,v^{\frac 1p}\|_{L^p(\R^n)}
\\
{}\leq{}&c_{1}\|e^{-\tau x_n}\nabla \vec F\, \vec v^{\frac 1p}\|_{L^p(S_\epsilon)} + c_{1}\|e^{-\tau x_n}\nabla \vec F\, v^{\frac 1p}\|_{L^p(\{x_n>\epsilon\})}
\\
{}\leq{}&c_{1}C\|e^{-\tau x_n} \G \vec F \,v^{\frac 1p}\|_{L^p(S_\epsilon)} + c_{1}e^{-\tau \epsilon }\|\nabla \vec F v^{\frac 1p}\|_{L^p(\{x_n>\epsilon\})}
\\
{}\leq{}& c_{1}C\sum_{j=1}^N \|e^{-\tau x_n} |\G_j| f_j \,v^{\frac 1p}\|_{L^p(S_\epsilon)} + c_{1}e^{-\tau \epsilon }\|\nabla \vec F v^{\frac 1p}\|_{L^p(\{x_n>\epsilon\})}
\\
{}\leq{}&c_{1}C\sum_{j=1}^N \||\G_j | \,v^{\frac 1p} u^{-\frac
1q}\|_{L^r(S_\epsilon\cap \supp \vec F)}\|e^{-\tau x_n} f_ju^{\frac 1q}\|_{L^q(S^\epsilon)} +
C'e^{-\tau \epsilon }
\\
{}\leq{}&c_{1}CN\||\G| \,v^{\frac 1p} u^{-\frac
1q}\|_{L^r(S_\epsilon\cap \supp \vec F)}\|e^{-\tau x_n} |\vec F| u^{\frac
1q}\|_{L^q(S^\epsilon)} + C'e^{-\tau \epsilon
},
\end{align*}
where we have let $C'=c_{1}\|\nabla \vec F\, v^{\frac
1p}\|_{L^p(\{x_n>\epsilon\})}$.

 We chose $\epsilon>0$ so that $c_{1}CN\||\G| \,v^{\frac 1p} u^{-\frac
1q}\|_{L^r(S_\epsilon\cap \supp \vec F)} <\frac 12$.
 We gather
 $$
 \|e^{-\tau x_n} \vec F u^{\frac 1q}\|_{L^q(S^\epsilon)}\leq
 \frac 12\,\|e^{-\tau x_n} \vec F u^{\frac 1q}\|_{L^q(S^\epsilon)}+ C'e^{-\tau \epsilon}
 $$
 which gives
 $$
 \frac 12\,\|e^{ \tau (\epsilon-x_n)} \vec F u^{\frac 1q}\|_{L^q(S^\epsilon)} \leq
 C',
 $$
and we can conclude the proof as in Theorem \ref{T- UC-ellypt}.
\end{proof}

\medskip
\begin{Rem}
It is shown in \cite{LNP} that the $W^{1,1}(D, \, \R^n)$ solutions of the system $\nabla
\vec F= \G \vec F$, with $\G \in L^1(D, \R^{(n\times n)\times n})$, cannot
vanish on an open set. The  proof in \cite{LNP} does not use Carleman inequalities.
\end{Rem}

\subsection{The Dirac operator}\label{ssec-dirac}
Let $\alpha _{j}$, $j=0,\ldots,n$, be $N\times N$ matrices which satisfy the
following relations.
\begin{equation}\label{e1-Clifford}
\alpha_{j}^{*}=\alpha_{j},\quad
\alpha_{j}^{2}=I,\quad
\alpha_{j}\alpha _{k}+\alpha _{k}\alpha_{j}=0,\quad j\neq k
\end{equation}
(we also say that the $\alpha _{j} $ form a basis of a Clifford algebra). It is
known that for \eqref{e1-Clifford} to hold, $N$ must be in the form
$2^{[\frac{n+1}{2}]}m$, with $m>0 $ integer

The ($n$-dimensional) Dirac operator associated to the matrices $\alpha_j$ is a
matrix value operator, initially defined on $C^\infty_0(\R^n,\, \R^{N\times N})$
as follows.
\begin{equation*}\label{e2-dirac}
\mathcal{L}U=-i\sum_{j=1}^{n}\alpha _{j}\partial _{x_{j}}U.
\end{equation*}
Here, $\partial _{x_{i}}U$ is a matrix whose entries are the partial derivative of the entries of $U$.
We can use \eqref{e1-Clifford} to show that $\mathcal{L}\circ \mathcal{L} U=-\Delta U\, I$, where $I$ is the identity matrix.
When $U=fI$, where $f\in C^\infty_0(\R^n)$, we can see at once that
$(\mathcal{L}(fI))^2 = -I|\nabla f|^2,
$
 Thus, a Dirac operators can be viewed as a generalization of the gradient operator and a square root of the Laplacian.

There is a lot of literature on the Dirac operator and its role in several domains of mathematics and physics See e.g.
\cite{Cnops}. For example, the Dirac equation which describes free relativistic electrons
is represented by
$$
i\hbar \partial _{t}\psi(t,x)=H_{0}\psi(t,x),
$$
where $H_{0}$ is given explicitly by the $4\times 4$ matrix-valued differential
expression
$$
H_{0}=-i\hbar c\sum_{j=1}^{3}\alpha _{j}\partial _{x_{j}}
 +\alpha _{0}mc^2.
$$
Here, $c$ is the speed of light, $m$ is a mass of a particle and
$\hbar$ is the Planck's constant.

In \cite{DO} is proved that the the differential inequality
\begin{equation}\label{e-diff-Dirac}
|\mathcal{L}U|\leq |V U|
\end{equation}
 where $V(x)$ is a $N\times N$ matrix,
 has the strong unique continuation property from the origin whenever $V(x)\leq C|x|^{-1}$, with $0\leq C\leq 1$. It is also proved in \cite{DO} that the condition $C\leq 1$ cannot be improved. See also \cite{KY} and the references cited there.
 We prove the following

 \begin{Thm}\label{T- UC-Dir}
  Let $ f \in W_0^{1,p, v}(D)$ be a
  solution of the differential inequality
\eqref{e-diff-Dirac}. If $|\mathbf{V}|\in L^{r}(\supp f, \, u^{-\frac
rq}v^{\frac rp}\,dx )$ with $ \frac 1p=\frac 1r+\frac 1q$ and $ f$ vanishes on
one side of a hyperplane, then $f\equiv 0$.
\end{Thm}

\begin{proof} Since $\mathcal{L}(fI) \cdot \mathcal{L}(fI)= -I|\nabla f|^2$, we can see at once that $$|\nabla f|=|\mathcal{L}(fI) \cdot \mathcal{L} (fI)|\leq |\mathcal{L}(fI)|^2
$$
With this observation, the proof of Theorem \ref{T- UC-Dir} is almost a line-by-line repetition of the proof of Theorem \ref{T- UC-ellypt}.
We leave the details to the reader.
\end{proof}

\end{document}